\documentclass[oneside,11pt]{amsart}
\usepackage{amsmath, amsfonts,amsthm,times,graphics}
\usepackage{mathrsfs}
\usepackage[usenames]{color}
\usepackage[active]{srcltx}
 \makeatletter
\renewcommand*\subjclass[2][2000]{%
  \def\@subjclass{#2}%
  \@ifundefined{subjclassname@#1}{%
    \ClassWarning{\@classname}{Unknown edition (#1) of Mathematics
      Subject Classification; using '1991'.}%
  }{%
    \@xp\let\@xp\subjclassname\csname subjclassname@#1\endcsname
  }%
}
 \makeatother

\newtheorem*{ThmA}{Theorem A}
\newtheorem*{ThmB}{Theorem B}
\newtheorem*{ThmC}{Theorem C}

\newtheorem{Thm}{Theorem}[section]
\newtheorem{Cor}[Thm]{Corollary}
\newtheorem{Lem}[Thm]{Lemma}

\theoremstyle{definition}

\theoremstyle{remark}

\numberwithin{equation}{section}

\newcommand{\ee}{\mathrm{e}}

\theoremstyle{definition}

\def\be{\begin{equation}}
\def\ee{\end{equation}}

\newcommand{\ben}{\begin{enumerate}}
\newcommand{\een}{\end{enumerate}}

\newcommand{\br}{\begin{rem}}
\newcommand{\er}{\end{rem}}
\newcommand{\brs}{\begin{rems}}
\newcommand{\ers}{\end{rems}}
\newcommand{\bo}{\begin{obser}}
\newcommand{\eo}{\end{obser}}
\newcommand{\bos}{\begin{obsers}}
\newcommand{\eos}{\end{obsers}}
\newcommand{\bpf}{\begin{pf}}
\newcommand{\epf}{\end{pf}}
\newcommand{\ba}{\begin{array}}
\newcommand{\ea}{\end{array}}
\newcommand{\beq}{\begin{eqnarray}}
\newcommand{\beqq}{\begin{eqnarray*}}
\newcommand{\eeq}{\end{eqnarray}}
\newcommand{\eeqq}{\end{eqnarray*}}

\numberwithin{equation}{section}
\newcounter{minutes}
\divide\time by 60
\newcounter{hours}
\multiply\time by 60 \addtocounter{minutes}{-\time}
\begin{document}
\title{Remarks on ``Schwarz-type lemma, Landau-type theorem, and Lipschitz-type space of solutions to inhomogeneous biharmonic equations"}

\author[Shaolin Chen and  Hidetaka Hamada]{Shaolin Chen and Hidetaka Hamada}

%Shaolin Chen
\address{S. L.  Chen, College of Mathematics and
Statistics, Hengyang Normal University, Hengyang, Hunan 421002,
People's Republic of China; Hunan Provincial Key Laboratory of
Intelligent Information Processing and Application,  421002,
People's Republic of China.} \email{mathechen@126.com}

\address{H. Hamada, Faculty of Science and Engineering, Kyushu Sangyo University,
3-1 Matsukadai 2-Chome, Higashi-ku, Fukuoka 813-8503, Japan.}
\email{hi.hamada01@gmail.com}

%{\color{red}}

\maketitle

\def\thefootnote{}
%\footnotetext{* corresponding author}
%\footnotetext{The research of the second named author was supported by NNSF of China Grant Nos. 11501220, 11471128, NNSF of Fujian Province Grant No. 2016J01020,
%and the Promotion Program for Young and Middle-aged Teacher in Science and Technology Research of Huaqiao University (ZQN-PY402).}
\makeatletter\def\thefootnote{\@arabic\c@footnote}\makeatother

%Primary: 31A30; Secondary:  35J40

\begin{abstract}
Let $\varphi$, $\psi\in C(\mathbb{T})$, $g\in C(\overline{\mathbb{D}})$, where $\mathbb{D}$ and $\mathbb{T}$
denote the unit disk and the unit circle, respectively. Suppose that  $f\in C^{4}(\mathbb{D})$ satisfies
the following:
\begin{enumerate}
\item[{\rm (1)}] the inhomogeneous biharmonic equation $ \Delta(\Delta f(z))=g(z)$ for $z\in\mathbb{D}$,
\item[{\rm (2)}] the Dirichlet boundary conditions $\partial_{\overline{z}}f(\zeta)=\varphi(\zeta)$
and $f(\zeta)=\psi(\zeta)$ for $\zeta\in\mathbb{T}$.
\end{enumerate}
Recently, the authors in [J. Geom. Anal. 29: 2469-2491, 2019] showed that if   $\omega$ is a majorant with
$\limsup_{t\rightarrow0^{+}}\left(\omega(t)/t\right)<\infty$, $\psi=0$ and
$\varphi_1 \in\mathscr{L}_{\omega}(\mathbb{T})$, where $\varphi_1(e^{it})=\varphi(e^{it})e^{-it}$ for $t\in[0,2\pi]$,  then
$f\in\mathscr{L}_{\omega}(\mathbb{D})$. The purpose of this paper is to improve and generalize this result.
We not only prove that the condition ``$\limsup_{t\rightarrow0^{+}}\left(\omega(t)/t\right)<\infty$" is redundant,
 but also demonstrate that conditions ``$\psi=0$" and ``$\varphi_1\in\mathscr{L}_{\omega}(\mathbb{T})$" can be replaced by weaker conditions.
\end{abstract}

\maketitle \pagestyle{myheadings} \markboth{ S. L. Chen and H.
Hamada}{Hardy-Littlewood type Theorems of solutions to inhomogeneous biharmonic equations}

%Consequently, we solve a problem posed byT William T. Ross
%in  mathematical reviews of mathscinet on the paper \cite{K-2019}.
\maketitle
%\tableofcontents

\section{ Preliminaries and main results}\label{sec1}

Let $\mathbb{C}$ be the complex plane. For $a\in\mathbb{C}$ and $\rho>0$, let $\mathbb{D}(a,\rho)=\{z:~|z-a|<\rho\}$. In particular, we use $\mathbb{D}_{\rho}$ to
denote the disk $\mathbb{D}(0,\rho)$ and $\mathbb{D}$ to denote the unit disk $\mathbb{D}_{1}$. Moreover, let $\mathbb{T}:=\partial\mathbb{D}$ be the
unit circle. For $z=x+iy\in\mathbb{C}$, the complex formal differential operators
are defined by $\partial/\partial z=1/2(\partial/\partial x-i\partial/\partial y)$ and $\partial/\partial \overline{z}=1/2(\partial/\partial x+i\partial/\partial y)$, where $x,y\in\mathbb{R}$.
For
$\theta\in[0,2\pi]$, the  directional derivative of  a
complex-valued and differentiable function $f$ at $z\in\mathbb{D}$ is defined
by
$$\partial_{\theta}f(z):=\lim_{\rho\rightarrow0^{+}}\frac{f(z+\rho e^{i\theta})-f(z)}{\rho}=\partial_{z}f(z)e^{i\theta}
+\partial_{\overline{z}}f(z)e^{-i\theta},$$ where $\partial_{z}f:=\partial
f/\partial z,$ $\partial_{\overline{z}}f:=\partial f/\partial \overline{z}$
and $\rho$ is a positive real number such that $z+\rho
e^{i\theta}\in\mathbb{D}$. Then
$$\Lambda_{f}(z):=\max\{|\partial_{\theta}f(z)|:\; \theta\in[0,2\pi]\}=|\partial_{z}f(z)|+|\partial_{\overline{z}}f(z)|.
$$

Denote by $C^{m}(\Omega)$ the set of all complex-valued $m$-times continuously differentiable functions from $\Omega$
into $\mathbb{C}$, where $\Omega$ stands for a subset of $\mathbb{C}$ and $m\in\{0,1,2,\ldots\}$. In particular, let $C(\Omega):=C^{0}(\Omega)$
denote the set of all continuous functions in $\Omega$.

Let $\varphi$, $\psi\in C(\mathbb{T})$, $g\in C(\overline{\mathbb{D}})$ and $f\in C^{4}(\mathbb{D})$. Of particular interest to us is the following
inhomogeneous biharmonic equation in $\mathbb{D}$:
\be\label{eq-1.1} \Delta(\Delta f)=g,\ee
and the following its associated Dirichlet boundary value problem:
\begin{align}\label{eq-1.2}
\begin{cases}
\displaystyle \partial_{\overline{z}}f=\varphi, \;\;& \;\; \mbox{in}~\mathbb{T}, \\
\displaystyle f=\psi, \;\;&\;  \;\mbox{in}~\mathbb{T},
\end{cases}
\end{align}
where $$\Delta f:=\frac{\partial^{2} f}{\partial x^{2}}+\frac{\partial^{2} f}{\partial y^{2}}=4\frac{\partial^{2} f}{\partial z\partial\overline{z}}$$
is the Laplacian of $f$. In particular, if $g\equiv0$, then the
solutions to (\ref{eq-1.1}) are called biharmonic mappings (see
\cite{AA,Be,CWY,CLW,GMR-2018,MM,S-2003}).
For $z, w\in\mathbb{D}$, let
%\be\label{G-1}
$$
G(z,w)=|z-w|^{2}\log\left|\frac{1-z\overline{w}}{z-w}\right|^{2}-(1-|z|^{2})(1-|w|^{2})
$$
%\ee
and $$P(z,e^{i\theta})=\frac{1-|z|^{2}}{|1-ze^{-i\theta}|^{2}}$$
denote the  biharmonic Green function and the
Poisson kernel, respectively, where $\theta\in[0,2\pi]$.
It follows from \cite[Theorem 2]{Be} that all
solutions to the  inhomogeneous biharmonic equation (\ref{eq-1.1}) satisfying the boundary value conditions
(\ref{eq-1.2}) are given by

\beq\label{eq-ch-3}
f(z)&=&P[\psi](z)+Q[\psi](z)-(1-|z|^{2})P[\varphi_{1}](z)\\
\nonumber
&&-\frac{1}{16\pi}\int_{\mathbb{D}}g(w)G(z,w)dA(w),
\eeq where $dA(w)$ denotes the Lebesgue area measure in
$\mathbb{D}$,
$\varphi_{1}(e^{it})=\varphi(e^{it})e^{-it}$,
$$P[\psi](z)=\frac{1}{2\pi}\int_{0}^{2\pi}P(z,e^{it})\psi(e^{it})dt,$$
\be\label{Q}
Q[\psi](z)=\frac{1}{2\pi}\int_{0}^{2\pi}\overline{z}e^{it}\psi(e^{it})\frac{1-|z|^{2}}{(1-\overline{z}e^{it})^{2}}dt
\ee
and
$$P[\varphi_{1}](z)=\frac{1}{2\pi}\int_{0}^{2\pi}P(z,e^{it})\varphi_{1}(e^{it})dt.$$

The  biharmonic equation arises in areas of continuum mechanics,
including linear elasticity theory and the solution of Stokes flows
(cf. \cite{Ha,Kh-1996, We}). Most important applications of the theory of functions of one complex variable were obtained in the plane theory of elasticity and in the
approximate theory of plates subject to normal loading (cf. \cite{HB-1965, L-1964}).

A continuous increasing function $\omega:[0,\infty)\rightarrow[0,\infty)$ with $\omega(0)=0$ is called a majorant if
$\omega(t)/t$ is non-increasing for $t\in(0,\infty)$ (see \cite{Dy1,Dy2,P}).
Note that for any $c\geq 1$ and $t>0$, we have
$\omega(ct)\leq c\omega(t)$ for every majorant $\omega$.
For $\nu_{0}>0$ and $0<\nu<\nu_{0}$, we consider the
following conditions on a majorant $\omega$:
\be\label{eq2x}
\int_{0}^{\nu}\frac{\omega(t)}{t}\,dt\leq\, M\omega(\nu)
\ee
and
\be\label{eq3x}
\nu\int_{\nu}^{\infty}\frac{\omega(t)}{t^{2}}\,dt\leq\, M
\omega(\nu),
\ee
where $M$ denotes a positive constant. A majorant $\omega$ is henceforth  called fast (resp. slow) if condition (\ref{eq2x}) (resp. (\ref{eq3x}) ) is fulfilled.
In particular, a majorant $\omega$ is said to be  regular if it satisfies the
conditions (\ref{eq2x}) and (\ref{eq3x}) (see \cite{Dy1,Dy2}).

Given a majorant $\omega$ and a subset  $\Omega$ of
$\mathbb{C}$, a function $f$ of  $\Omega$ into $\mathbb{C}$ is
said to belong to the Lipschitz space
$\mathscr{L}_{\omega}(\Omega)$
 if there is a positive constant $M$ such that
 \be\label{rrt-1}|f(z_{1})-f(z_{2})| \leq\,M\omega\left(|z_{1}-z_{2}|\right), \quad
z_{1},z_{2} \in \Omega.\ee
Furthermore, let
$$\|f\|_{\mathscr{L}_{\omega}(\Omega),s}:=\sup_{z_{1},z_{2}\in\Omega,z_{1}\neq\,z_{2}}\frac{|f(z_{1})-f(z_{2})|}{\omega(|z_{1}-z_{2}|)}<\infty.$$
Note that if $\Omega$ is a proper subdomain of $\mathbb{C}$ and
$f\in \mathscr{L}_{\omega}(\Omega)$, then $f$ is continuous on
$\overline{\Omega}$ and (\ref{rrt-1}) holds for $z,w \in
\overline{\Omega}$ (see \cite{Dy2}).
  In particular, we say that
a function $f$ belongs to the  local Lipschitz space
$\mbox{loc}\mathscr{L}_{\omega}(\Omega)$ if (\ref{rrt-1})
%$$|f(z)-f(w)|\leq C\omega\left(|z-w|\right)$$
holds, with a fixed
positive constant $M$, whenever $z\in \Omega$ and
$|z-w|<\frac{1}{2}d_{\Omega}(z)$ (cf. \cite{Dy2,GM,L}), where $d_{\Omega}(z)$ is the
Euclidean distance between $z$ and the boundary of $\Omega$. Moreover, $\Omega$ is called an  $\mathscr{L}_{\omega}$-extension domain if
$\mathscr{L}_{\omega}(\Omega)=\mbox{loc}\mathscr{L}_{\omega}(\Omega).$
On the geometric
characterization of $\mathscr{L}_{\omega}$-extension domains, see
 \cite{GM}. In \cite{L}, Lappalainen
generalized the characterization of \cite{GM}, and proved that $\Omega$ is an
$\mathscr{L}_{\omega}$-extension domain if and only if each pair of points
$z_{1},z_{2}\in \Omega$ can be joined by a rectifiable curve $\gamma\subset \Omega$
satisfying
%\be\label{eq1.0}
$$
\int_{\gamma}\frac{\omega(d_{\Omega}(\zeta))}{d_{\Omega}(\zeta)}\,ds(\zeta)
\leq M\omega(|z_{1}-z_{2}|)
$$
%\ee
with some fixed positive constant
$M$, where $ds$ stands for the arc length measure on
$\gamma$.  Furthermore, Lappalainen \cite[Theorem 4.12]{L} proved
that $\mathscr{L}_{\omega}$-extension domains  exist only for fast majorants
$\omega$. In particular, $\mathbb{D}$ is an $\mathscr{L}_{\omega}$-extension domain of $\mathbb{C}$ for fast majorant $\omega$ (see \cite{Dy2}).

For complex-valued harmonic functions of $\mathbb{D}$ into $\mathbb{C}$, there is  a classical  Hardy-Littlewood theorem as follows (cf.
 \cite{H-L-31,H-L,Pav-2008}).
 \begin{ThmA}\label{H-L}{\rm (The Hardy-Littlewood's Theorem)}
 If $\varphi\in\Lambda_{\omega_{\beta}}(\mathbb{T})$, then
$P[\varphi]\in\Lambda_{\omega_{\beta}}(\overline{\mathbb{D}})$, where $\beta\in(0,1)$ and $\omega_{\beta}(t)=t^{\beta}$ for $t\geq0$.
 \end{ThmA}

On the unit disk $\mathbb{D}$,
Nolder and Oberlin \cite[Lemma 4.2]{No} gave the following necessary and sufficient condition for a majorant $\omega$
to hold the above  Hardy and Littlewood type theorem (cf.\cite[Lemma 4]{Dy1}).
Chen, Hamada and Xie \cite{CHX} generalized the following result to the Euclidean unit ball of $\mathbb{R}^n$,
$n\geq 2$.

\begin{ThmB}\label{H-L-iff}
Let $\omega$ be a majorant.
The following statements are equivalent.
\begin{enumerate}
\item[(1)]
There exists a positive constant $M$ such that
\be\label{HL-condition}
\lambda\int_{\lambda}^{\pi}\frac{\omega(t)}{t^2}dt\leq M\omega(\lambda)
\ee
for all $\lambda\in (0,\pi]$.

\item[(2)]
If $\varphi\in\mathscr{L}_{\omega}(\mathbb{T})$, then
 $P[\varphi]\in \mathscr{L}_{\omega}(\mathbb{D})$.
\end{enumerate}
\end{ThmB}

In \cite{CLW}, the authors investigated the Hardy-Littlewood type theorem for solutions  to inhomogeneous biharmonic equations as follows.
For a function $\varphi$ of $\mathbb{T}$ into $\mathbb{C}$,
let
$\varphi_{1}(e^{it})=\varphi(e^{it})e^{-it}$ for $t\in [0,2\pi]$.

\begin{ThmC}\label{Thm-5} Suppose that $\omega$ is a majorant and
\be\label{remove-1}\limsup_{t\rightarrow0^{+}}\frac{\omega(t)}{t}=c<\infty,\ee and
suppose that $f\in C^{4}(\mathbb{D})$  satisfies the following equations:
$$\begin{cases}\displaystyle \Delta(\Delta f)=g &\mbox{ in }\, \mathbb{D},\\
\displaystyle  f_{\overline{z}}=\varphi &\mbox{ in }\, \mathbb{T},\\
\displaystyle f=0&\mbox{ in }\, \mathbb{T},
\end{cases}$$ where $g\in C(\overline{\mathbb{D}})$ and
$\varphi_{1}\in \mathscr{L}_{\omega}(\mathbb{T})$.
 Then $f\in\mathscr{L}_{\omega}(\mathbb{D})$.
\end{ThmC}

The main aim of this paper is to establish some Hardy-Littlewood type theorems for solutions  to inhomogeneous biharmonic equations without the additional assumptions ``(\ref{remove-1}) and $f=0$ in $\mathbb{T}$".
The main results are as follows.

\begin{Thm}\label{thm-1}
Let $\omega_{1}$ be a  majorant and $\omega_{2}$ be a majorant which satisfies \eqref{HL-condition}.
Suppose that $g\in C(\overline{\mathbb{D}})$, $\varphi\in\mathscr{L}_{\omega_{1}}(\mathbb{T})$ and $\psi\in\mathscr{L}_{\omega_{2}}(\mathbb{T})$,
and suppose that
$f\in C^{4}(\mathbb{D})$ satisfies the following equation:
\be\label{equat-1}\begin{cases}\displaystyle \Delta(\Delta f)=g &\mbox{ in }\, \mathbb{D},\\
\displaystyle  \partial_{\overline{z}}f=\varphi &\mbox{ in }\, \mathbb{T},\\
\displaystyle f=\psi&\mbox{ in }\, \mathbb{T}.
\end{cases}\ee
Then there is a positive constant $M$ such that, for all $z_{1}, z_{2}\in\mathbb{D}$,
\be\label{Lip-1}|f(z_{1})-f(z_{2})|\leq M\omega_{1}(|z_{1}-z_{2}|)+M\omega_{2}(|z_{1}-z_{2}|).\ee
\end{Thm}

The following result easily follows from Theorem \ref{thm-1}.

\begin{Cor}\label{cor-1}
Let $\omega$ be a  majorant.
Suppose that $g\in C(\overline{\mathbb{D}})$, $\varphi\in\mathscr{L}_{\omega}(\mathbb{T})$
and
$f\in C^{4}(\mathbb{D})$ satisfies the following equation:
$$\begin{cases}\displaystyle \Delta(\Delta f)=g &\mbox{ in }\, \mathbb{D},\\
\displaystyle  \partial_{\overline{z}}f=\varphi &\mbox{ in }\, \mathbb{T},\\
\displaystyle f=0&\mbox{ in }\, \mathbb{T}.
\end{cases}$$
Then $f\in\mathscr{L}_{\omega}(\mathbb{D})$.
\end{Cor}

%In view of Lemma \ref{varphi-Lipschitz}, $\varphi\in\mathscr{L}_{\omega_{1}}(\mathbb{T})$ in Theorem \ref{thm-1}
%can be replaced by $\varphi_{1}\in\mathscr{L}_{\omega_{1}}(\mathbb{T})$.
If we replace $\varphi\in\mathscr{L}_{\omega_{1}}(\mathbb{T})$ by $\varphi_{1}\in\mathscr{L}_{\omega_{1}}(\mathbb{T})$ in Theorem \ref{thm-1}, then we get the following result.

\begin{Thm}\label{thm-2}
Let $\omega_{1}$ be a  majorant and $\omega_{2}$ be a majorant which satisfies \eqref{HL-condition}.
Suppose that $g\in C(\overline{\mathbb{D}})$, $\varphi_{1}\in\mathscr{L}_{\omega_{1}}(\mathbb{T})$ and $\psi\in\mathscr{L}_{\omega_{2}}(\mathbb{T})$,
and suppose that
$f\in C^{4}(\mathbb{D})$ satisfies the following equation:
%\be\label{equat-2}
$$
\begin{cases}\displaystyle \Delta(\Delta f)=g &\mbox{ in }\, \mathbb{D},\\
\displaystyle  \partial_{\overline{z}}f=\varphi &\mbox{ in }\, \mathbb{T},\\
\displaystyle f=\psi&\mbox{ in }\, \mathbb{T}.
\end{cases}
$$
%\ee
Then there is a positive constant $M$ such that, for all $z_{1}, z_{2}\in\mathbb{D}$,
%\be\label{Lip-1-b}
$$
|f(z_{1})-f(z_{2})|\leq M\omega_{1}(|z_{1}-z_{2}|)+M\omega_{2}(|z_{1}-z_{2}|).
$$
%\ee
\end{Thm}

%We remark that Theorem \ref{thm-2} is an improvement of Theorem
%\Ref{Thm-5}.
In particular,
if we take $\psi\equiv0$ in Theorem \ref{thm-2}, then we get the following result.

\begin{Cor}\label{cor-2}
Let $\omega$ be a  majorant.
Suppose that $g\in C(\overline{\mathbb{D}})$, $\varphi_{1}\in\mathscr{L}_{\omega}(\mathbb{T})$
and
$f\in C^{4}(\mathbb{D})$ satisfies the following equation:
$$\begin{cases}\displaystyle \Delta(\Delta f)=g &\mbox{ in }\, \mathbb{D},\\
\displaystyle  \partial_{\overline{z}}f=\varphi &\mbox{ in }\, \mathbb{T},\\
\displaystyle f=0&\mbox{ in }\, \mathbb{T}.
\end{cases}$$
Then $f\in\mathscr{L}_{\omega}(\mathbb{D})$.
\end{Cor}

The proofs of Theorems \ref{thm-1} and \ref{thm-2}
will be presented in Sec. \ref{sec2}.

\section{The proofs of the main results}\label{sec2}
In this section, we first present some lemmas that will play an important role in proving Theorem \ref{thm-1}.

\begin{Lem}
\label{J4}
Let $\omega_{2}$ be a majorant which satisfies \eqref{HL-condition}
and $\psi\in\mathscr{L}_{\omega_{2}}(\mathbb{T})$.
For $z=re^{i\theta}\in \mathbb{D}$,
let
\be\label{J4-def}
J_1(re^{i\theta})=\frac{1}{2\pi}\int_{0}^{2\pi}\frac{|\psi(e^{it})-\psi(e^{i\theta})|}{|1-\overline{z}e^{it}|^{2}}dt.
\ee
Then there exists a positive constant $M$ such that
\begin{align*}
J_1(re^{i\theta})\leq M\|\psi\|_{\mathscr{L}_{\omega_{2}}(\mathbb{T}),s}\frac{\omega_{2}(1-r)}{1-r},
\quad
|z|=r<1,
\quad
\theta\in [0, 2\pi].
\end{align*}
\end{Lem}

\begin{proof}
\noindent $\mathbf{Case~1.}$ We first assume $r\geq1/2$.\\
Let $E_{1}(\theta)=\{t\in[-\pi+\theta, \pi+\theta]:~|t-\theta|\leq1-r\}$ and $E_{2}(\theta)=[-\pi+\theta, \pi+\theta]\backslash E_{1}(\theta)$.
Then
\beqq
J_1(z)&=&\frac{1}{2\pi}\int_{E_{1}\cup E_{2}}\frac{|\psi(e^{it})-\psi(e^{i\theta})|}{|1-\overline{z}e^{it}|^{2}}dt\\
&\leq&\frac{\|\psi\|_{\mathscr{L}_{\omega_{2}}(\mathbb{T}),s}}{2\pi}\int_{E_{1}\cup E_{2}}\frac{\omega_{2}(|e^{i\theta}-e^{it}|)}{|1-\overline{z}e^{it}|^{2}}dt\\
&=&\frac{\|\psi\|_{\mathscr{L}_{\omega_{2}}(\mathbb{T}),s}}{2\pi}\int_{E_{1}\cup E_{2}}\frac{\omega_{2}\left(2\left|\sin\frac{\theta-t}{2}\right|\right)}{(1-r)^{2}+4r\left(\sin\frac{\theta-t}{2}\right)^{2}}dt,
\eeqq
which, together with $\vartheta \geq \sin \vartheta\geq2\vartheta/\pi$ for $0\leq \vartheta\leq\pi/2$, implies that
\beq\label{eq-1.41}
J_1(z)&\leq&\frac{\|\psi\|_{\mathscr{L}_{\omega_{2}}(\mathbb{T}),s}}{2\pi}\int_{E_{1}\cup E_{2}}\frac{\omega_{2}(|\theta-t|)}{(1-r)^{2}+\frac{4}{\pi^{2}}r\left(\theta-t\right)^{2}}dt\\ \nonumber
&\leq&\frac{\|\psi\|_{\mathscr{L}_{\omega_{2}}(\mathbb{T}),s}}{2\pi}\left(
\int_{E_{1}}\frac{\omega_{2}(1-r)}{(1-r)^{2}}dt+
\frac{\pi^2}{2}\int_{E_{2}}\frac{\omega_{2}(|\theta-t|)}{\left(\theta-t\right)^{2}}dt\right)\\ \nonumber
&\leq&\frac{\|\psi\|_{\mathscr{L}_{\omega_{2}}(\mathbb{T}),s}}{\pi}\left(\frac{\omega_{2}(1-r)}{1-r}+\frac{\pi^2}{2}\int_{1-r}^{\pi}\frac{\omega_{2}(\rho)}{\rho^{2}}d\rho\right).
\eeq
By the assumption \eqref{HL-condition}, we see that there is a positive constant $M$ such that $$\int_{1-r}^{\pi}\frac{\omega_{2}(\rho)}{\rho^{2}}d\rho\leq M\frac{\omega_{2}(1-r)}{1-r},$$
which, together with (\ref{eq-1.41}), yields that
\be\label{eq-9}J_1(z)\leq\|\psi\|_{\mathscr{L}_{\omega_{2}}(\mathbb{T}),s}\left(\frac{1}{\pi}+\frac{\pi}{2}M\right)\frac{\omega_{2}(1-r)}{1-r}.\ee

\noindent $\mathbf{Case~2.}$ We  assume $r\leq1/2$.\\
\beq\label{eq-10}
J_1(z)&\leq&\frac{\|\psi\|_{\mathscr{L}_{\omega_{2}}(\mathbb{T}),s}}{2\pi}\int_{0}^{2\pi}\frac{\omega_{2}(|e^{i\theta}-e^{it}|)}{|1-\overline{z}e^{it}|^{2}}dt\\ \nonumber
&\leq&
\|\psi\|_{\mathscr{L}_{\omega_{2}}(\mathbb{T}),s}\frac{\omega_{2}(2)}{(1-r)^{2}}\\ \nonumber
&\leq&8\|\psi\|_{\mathscr{L}_{\omega_{2}}(\mathbb{T}),s}\frac{\omega_{2}(1-r)}{1-r}.
\eeq
(\ref{eq-9}) and (\ref{eq-10}) give the desired result.
This completes the proof.
\end{proof}
\begin{Lem}
\label{Lipshitz-lemma}
Let $\omega_{2}$ be a majorant which satisfies \eqref{HL-condition}
and $\psi\in\mathscr{L}_{\omega_{2}}(\mathbb{T})$.
Let $Q[\psi]$  be as in \eqref{Q}.
Then
there is a positive constant $M$ such that
\be
\label{eq-Lipschitz}
|Q[\psi](r\xi)|
\leq
 M\|\psi\|_{\mathscr{L}_{\omega_{2}}(\mathbb{T}),s}\omega_2(1-r)
\ee
for all $\xi \in \mathbb{T}$ and $r\in [0,1)$.
\end{Lem}

\begin{proof}
Let $z=re^{i\theta}\in\mathbb{D}$, where $r\in[0,1)$ and $\theta\in[0,2\pi]$.
Since $$\frac{1}{2\pi}\int_{0}^{2\pi}\overline{z}e^{it}\psi(e^{i\theta})\frac{1-|z|^{2}}{(1-\overline{z}e^{it})^{2}}dt=0,$$
we have
\begin{align*}
|Q[\psi](z)|&=\left|\frac{1}{2\pi}\int_{0}^{2\pi}\overline{z}e^{it}(\psi(e^{it})-\psi(e^{i\theta}))\frac{1-|z|^{2}}{(1-\overline{z}e^{it})^{2}}dt\right|
\\
&\leq
(1-|z|^2)J_1(z),
\end{align*}
which combined with Lemma \ref{J4} implies
\eqref{eq-Lipschitz}, as desired.
This completes the proof.
\end{proof}

\begin{Lem}
\label{Q-derivative}
Let $\omega_{2}$ be a majorant which satisfies \eqref{HL-condition}
and $\psi\in\mathscr{L}_{\omega_{2}}(\mathbb{T})$.
Let $Q[\psi]$  be as in \eqref{Q}.
Then
there is a positive constant $M$ such that
\[
\Lambda_{Q[\psi]}(re^{i\theta})\leq M\|\psi\|_{\mathscr{L}_{\omega_{2}}(\mathbb{T}),s}\frac{\omega_{2}(1-r)}{1-r},
\quad
|z|=r<1,
\quad
\theta\in [0, 2\pi].
\]
\end{Lem}

\begin{proof}
Let $z=re^{i\theta}\in\mathbb{D}$, where $r\in[0,1)$ and $\theta\in[0,2\pi]$.
Since $$\frac{1}{2\pi}\int_{0}^{2\pi}\psi(e^{i\theta})\frac{\overline{z}e^{it}}{(1-\overline{z}e^{it})^{2}}dt=0,$$
we see that
\beqq
\partial_{z}Q[\psi](z)=-\frac{1}{2\pi}\int_{0}^{2\pi}(\psi(e^{it})-\psi(e^{i\theta}))\frac{\overline{z}^{2}e^{it}}{(1-\overline{z}e^{it})^{2}}dt
\eeqq
and
\beqq
\partial_{\overline{z}}Q[\psi](z)&=&-\frac{1}{2\pi}\int_{0}^{2\pi}(\psi(e^{it})-\psi(e^{i\theta}))\frac{|z|^{2}e^{it}}{(1-\overline{z}e^{it})^{2}}dt\\
&&+\frac{1-|z|^{2}}{2\pi}\int_{0}^{2\pi}(\psi(e^{it})-\psi(e^{i\theta}))\frac{e^{it}(1+\overline{z}e^{it})}{(1-\overline{z}e^{it})^{3}}dt.
\eeqq
Then
\be\label{eq-5}
|\partial_{z}Q[\psi](z)|\leq\frac{|z|^{2}}{2\pi}\int_{0}^{2\pi}\frac{|\psi(e^{it})-\psi(e^{i\theta})|}{|1-\overline{z}e^{it}|^{2}}dt\leq
J_1(z)
\ee
and
\be\label{eq-6}
|\partial_{\overline{z}}Q[\psi](z)|\leq\frac{|z|^{2}+(1+|z|)^{2}}{2\pi}\int_{0}^{2\pi}\frac{|\psi(e^{it})-\psi(e^{i\theta})|}{|1-\overline{z}e^{it}|^{2}}dt
\leq5J_1(z),
\ee where $J_1$ is as in \eqref{J4-def}.
Combining \eqref{eq-5}, \eqref{eq-6} and Lemma \ref{J4},
we obtain that
\[
\Lambda_{Q[\psi]}(re^{i\theta})\leq M\|\psi\|_{\mathscr{L}_{\omega_{2}}(\mathbb{T}),s}\frac{\omega_{2}(1-r)}{1-r},
\quad
|z|=r<1,
\quad
\theta\in [0, 2\pi].
\]
This completes the proof.
\end{proof}

For $z, w\in \mathbb{C}$,
let $[z, w]$ denote the segment from $z$ to $w$.

\begin{Lem}
\label{Lipshitz2-lemma}
Let $\omega_{2}$ be a majorant which satisfies \eqref{HL-condition}
and $\psi\in\mathscr{L}_{\omega_{2}}(\mathbb{T})$.
Let $Q[\psi]$  be as in \eqref{Q}.
Then
there is a positive constant $M$ such that
\[
%\be\label{eq-Lipschitz2}
|Q[\psi](r_1\xi)-Q[\psi](r_2\xi)|
\leq
M\|\psi\|_{\mathscr{L}_{\omega_{2}}(\mathbb{T}),s}\omega_2(r_1-r_2)
\]
%\ee
for all $\xi \in \mathbb{T}$ and $r_1, r_2$
with $0\leq r_2<r_1<1$.
\end{Lem}

\begin{proof}
First, assume that $1-r_1\leq r_1-r_2$.
Then, we have $$1-r_2=(1-r_1)+(r_1-r_2)\leq 2(r_1-r_2).$$
Therefore, by Lemma \ref{Lipshitz-lemma},
there is a positive constant $M$ such that
\beqq
|Q[\psi](r_1\xi)-Q[\psi](r_2\xi)|
&\leq&
|Q[\psi](r_1\xi)|+|Q[\psi](r_2\xi)|
\\
&\leq&
M\|\psi\|_{\mathscr{L}_{\omega_{2}}(\mathbb{T}),s}(\omega(1-r_1)+\omega(1-r_2))
\\
&\leq &
M\|\psi\|_{\mathscr{L}_{\omega_{2}}(\mathbb{T}),s}(\omega(r_1-r_2)+\omega(2(r_1-r_2))
\\
&\leq&
3M\|\psi\|_{\mathscr{L}_{\omega_{2}}(\mathbb{T}),s}\omega(r_1-r_2).
\eeqq

Next, we consider the case $1-r_1> r_1-r_2$.
Since $\frac{1-r_1}{r_1-r_2}>1$,
by using Lemma \ref{Q-derivative},
there is a positive constant $M$ such that
\beqq
|Q[\psi](r_1\xi)-Q[\psi](r_2\xi)|
&\leq&
\int_{[r_{1}\xi,r_{2}\xi]}\Lambda_{Q[\psi]}(z)|dz|
\\
&\leq&
M\|\psi\|_{\mathscr{L}_{\omega_{2}}(\mathbb{T}),s}\frac{r_1-r_2}{1-r_1}\omega_2(1-r_1)
\\
&\leq&
M\|\psi\|_{\mathscr{L}_{\omega_{2}}(\mathbb{T}),s}\omega_2(r_1-r_2).
\eeqq
This completes the proof.
\end{proof}

\begin{Lem}
\label{Lipshitz3-lemma}
Let $\omega_{2}$ be a majorant which satisfies \eqref{HL-condition}
and $\psi\in\mathscr{L}_{\omega_{2}}(\mathbb{T})$.
Let $Q[\psi]$  be as in \eqref{Q}.
Then
there is a positive constant $M$ such that
%\be\label{eq-Lipschitz3}
\[
|Q[\psi](z_1)-Q[\psi](z_2)|
\leq
M\|\psi\|_{\mathscr{L}_{\omega_{2}}(\mathbb{T}),s}\omega_2(|z_1-z_2|)
\]
%\ee
for all $z_1, z_2 \in \mathbb{D}$ with $|z_1|=|z_2|$.
\end{Lem}

\begin{proof}
First, assume that $1-|z_1|\leq |z_1-z_2|$.
Then, in view of Lemma \ref{Lipshitz-lemma},
we deduce that
there is a positive constant $M$ such that
\beqq
|Q[\psi](z_1)-Q[\psi](z_2)|
&\leq&
|Q[\psi](z_1)|+|Q[\psi](z_2)|
\\
&\leq&
M\|\psi\|_{\mathscr{L}_{\omega_{2}}(\mathbb{T}),s}(\omega_2(1-|z_1|)+\omega_2(1-|z_2|))
\\
&\leq&
2M\|\psi\|_{\mathscr{L}_{\omega_{2}}(\mathbb{T}),s}\omega_2(|z_1-z_2|).
\eeqq

Next, we consider the case $1-|z_1|>|z_1-z_2|$.
Since $\frac{1-|z_1|}{|z_1-z_2|}>1$,
by using Lemma \ref{Q-derivative},
there is a positive constant $M$ such that
\beqq
|Q[\psi](z_1)-Q[\psi](z_2)|
&\leq&
\int_{[z_1,z_2]}\Lambda_{Q[\psi]}(z)|dz|
\\
&\leq&
M\|\psi\|_{\mathscr{L}_{\omega_{2}}(\mathbb{T}),s}\frac{|z_1-z_2|}{1-|z_1|}\omega_2(1-|z_1|)
\\
&\leq&
M\|\psi\|_{\mathscr{L}_{\omega_{2}}(\mathbb{T}),s}\omega_2(|z_1-z_2|).
\eeqq
This completes the proof.
\end{proof}

\subsection{The proof of Theorem \ref{thm-1}} It follows from (\ref{eq-ch-3}) that the solutions to  (\ref{equat-1}) are given by

\be\label{eq-ch-t10}
f(z)=P[\psi](z)+Q[\psi](z)-J_{2}(z)-J_{3}(z),
\ee
where  $Q[\psi]$ is as in \eqref{Q},
$$J_{2}(z)=\frac{1-|z|^{2}}{2\pi}\int_{0}^{2\pi}P(z,e^{it})\varphi(e^{it})e^{-it}dt$$ and $$J_{3}(z)=\frac{1}{16\pi}\int_{\mathbb{D}}g(w)G(z,w)dA(w).$$

\noindent $\mathbf{Step~1.}$ We first show that there is a positive constant $M$  such that
\be\label{eq-q1}| Q[\varphi](z_{1})- Q[\varphi](z_{2})|\leq
 M\omega_{2}(|z_{1}-z_{2}|),~z_{1},~z_{2}\in\mathbb{D}.\ee

We may assume that $|z_1|\geq |z_2|$ with $z_1\neq z_2$.
If $z_2=tz_1$ for some $t\in [0,1]$, then the result follows from Lemma \ref{Lipshitz2-lemma}.
So, it suffices to consider the case $z_2$ is not contained in the segment between $0$ and $z_1$.
Let $w=\frac{|z_1|}{|z_2|}z_2$.
Then
$$|z_2-w|\leq |z_1-z_2|~\mbox{and}~|z_1-w|\leq 2|z_1-z_2|,$$
which, together with Lemmas \ref{Lipshitz2-lemma} and \ref{Lipshitz3-lemma}, imply that there
is a positive constant $M$ such that
\beqq
\left| Q[\psi](z_1)-Q[\psi](z_2)\right|
&\leq&
\left| Q[\psi](z_1)-Q[\psi](w)\right|
+\left|Q[\psi](w)-Q[\psi](z_2)\right|
\\
&\leq&
M(\omega_2(|z_1-w|)+\omega_2(|w-z_2|))
\\
&\leq&
3M\omega_2(|z_1-z_2|),
\eeqq
which implies (\ref{eq-q1}).

\noindent $\mathbf{Step~2.}$
We will show that there is a positive constant $M$  such that
\be\label{eq-q2}|J_{2}(z_{1})-J_{2}(z_{2})|\leq M\omega_{1}(|z_{1}-z_{2}|),~z_{1},~z_{2}\in\mathbb{D}.\ee

Elementary calculations lead to
\[
\partial_{z}P(z,e^{it})=\frac{e^{it}}{(e^{it}-z)^2}
\]
and
\[
\partial_{\overline{z}}P(z,e^{it})=\frac{e^{-it}}{(e^{-it}-\overline{z})^2},
\]
which imply that
\beqq
\Lambda_{J_{2}}(z)&\leq&\|\varphi\|_{\infty}\frac{|z|}{\pi}\int_{0}^{2\pi}P(z,e^{it})dt\\
&&+\|\varphi\|_{\infty}\frac{1-|z|^{2}}{2\pi}\int_{0}^{2\pi}\left(|\partial_{z}P(z,e^{it})|+|\partial_{\overline{z}}P(z,e^{it})|\right)dt\\
&=&\|\varphi\|_{\infty}\frac{1+|z|}{\pi}\int_{0}^{2\pi}P(z,e^{it})dt\\
&\leq&
4\|\varphi\|_{\infty},
\eeqq
where $\|\varphi\|_{\infty}=\sup_{\zeta\in \mathbb{T}}|\varphi(\zeta)|$.
Then, in view of $\omega_{1}(|z_{1}-z_{2}|)/|z_{1}-z_{2}|\geq\omega_{1}(2)/2$, we have
\beqq
|J_{2}(z_{1})-J_{2}(z_{2})|&\leq&\int_{[z_{1},z_{2}]}\Lambda_{J_{2}}(z)|dz|
\\
&\leq&
4\|\varphi\|_{\infty}|z_{1}-z_{2}|\\
&\leq&
\frac{8\|\varphi\|_{\infty}}{\omega_{1}(2)}\omega_{1}(|z_{1}-z_{2}|).
\eeqq

\noindent $\mathbf{Step~3.}$ At last, we prove  that there is a positive constant $M$  such that
\be\label{eq-q3}|J_{3}(z_{1})-J_{3}(z_{2})|\leq M\omega_{1}(|z_{1}-z_{2}|),~z_{1},~z_{2}\in\mathbb{D}.\ee
By \cite[Lemma 2.5]{LP}, we have
\be\label{eq-o3}
\Lambda_{J_{3}}(z)\leq\frac{23}{48}\|g\|_{\infty},
\ee
where $\|g\|_{\infty}=\sup_{z\in\mathbb{D}}|g(z)|$.
Then, by (\ref{eq-o3}), we have
\beqq
|J_{3}(z_{1})-J_{3}(z_{2})|\leq\int_{[z_{1},z_{2}]}\Lambda_{J_{3}}(z)|dz|\leq\frac{23}{24\omega_{1}(2)}\|g\|_{\infty}\omega_{1}(|z_{1}-z_{2}|).
\eeqq
Therefore, (\ref{Lip-1}) follows from (\ref{eq-ch-t10}), (\ref{eq-q1}), (\ref{eq-q2}), (\ref{eq-q3}) and Theorem B.
The proof of this theorem is complete.
\qed

\begin{Lem}
\label{varphi-Lipschitz}
Let $\omega$ be a majorant.
Then, for a function $\varphi$ of $\mathbb{T}$ into $\mathbb{C}$,
$\varphi_{1}\in \mathscr{L}_{\omega}(\mathbb{T})$ if and only if $\varphi\in \mathscr{L}_{\omega}(\mathbb{T})$.
\end{Lem}

\begin{proof}
First, assume that $\varphi\in \mathscr{L}_{\omega}(\mathbb{T})$.
Then, for $\xi, \zeta \in \mathbb{T}$, we have
\begin{align*}
|\varphi_1(\xi)-\varphi_1(\zeta)|
&=
|\varphi(\xi)\overline{\xi}-\varphi(\zeta)\overline{\zeta}|
\\
&\leq
|\varphi(\xi)-\varphi(\zeta)|+|\varphi(\zeta)| |\overline{\xi}-\overline{\zeta}|
\\
&\leq
M\omega(|\xi-\zeta|)+\|\varphi\|_{\infty}\frac{2}{\omega(2)}\omega(|\xi-\zeta|),
\end{align*}
which implies that $\varphi_{1}\in \mathscr{L}_{\omega}(\mathbb{T})$.

The proof of the converse is similar.
This completes the proof.
\end{proof}

\subsection{The proof of Theorem \ref{thm-2}}
The proof of this theorem follows from Theorem \ref{thm-1}  and Lemma \ref{varphi-Lipschitz}.
\qed

\section*{Statements and Declarations}

%\subsection*{Acknowledgments}
\subsection*{Funding}
This work of the second author was supported by
Japan Society for the Promotion of Science KAKENHI Grant Number JP22K03363.

\subsection*{Data Availability Statement}
Data sharing not applicable to this article as no datasets were generated or analysed during the current study.

\subsection*{Conflict of interest}
The authors have no Conflict of interest to declare that are relevant to the content of this article.

%For $p\in(0,\infty]$, the  generalized Bergman space
%$\mathscr{B}^{p}_{G}(\mathbb{D})$ consists of all measurable functions $f:\;\mathbb{D}\rightarrow\mathbb{C}$ such that
%$$\|f\|_{b^{p}}=
%\begin{cases}
%\displaystyle\left(\int_{\mathbb{D}}|f(z)|^{p}d\sigma(z)\right)^{\frac{1}{p}}
%& \mbox{if } p\in(0,\infty),\\
%\displaystyle \mbox{ess}\sup\{|f(z)|:\; z\in \mathbb{D}\} &\mbox{if } p=\infty,
%\end{cases}
%$$
%where  $d\sigma(z)=\frac{1}{\pi}dxdy$ denotes  the normalized Lebesgue area measure on $\mathbb{D}$.
%The classical Bergman space $\mathscr{B}^{p}(\mathbb{D})$, that is, all the elements are analytic, is a subspace of $\mathscr{B}^{p}_{G}%(\mathbb{D})$.
%Obviously, $\mathscr{H}^{p}(\mathbb{D})\subset\mathscr{B}^{p}(\mathbb{D})$ for each $p\in(0,\infty]$.

%As an application of Theorem \ref{Ch-H-Z-1}, together with \cite[Lemma 1.3]{Pav-2019} or \cite[Theorem 4.17]{Z1}, we have the following corollary.

%\begin{Cor}\label{8-3-Cor}
%For a fixed $p\in(1,+\infty)$, let $f=h+\bar{g}\in\mathscr{H}_{H}^{p}(\mathbb{D})$. Then, for any $z\in\mathbb{D}$, the following inequality holds:
%\be\label{8-4-cor-ineq}|f(z)|\leq(1-|z|^2)^{-1/p}\big(\|h\|_p+\|g\|_p\big).\ee
%Moreover, the equality in {\rm (\ref{8-4-cor-ineq})} occurs if
%$$f_{a}(z)=C\left(\frac{1-|a|^2}{(1-\bar{a}z)^2}\right)^{1/p},~z\in\mathbb{D},$$
% where $C$ is a constant and $a\in \mathbb{D}$ is fixed.
%\end{Cor}

\end{document}